\documentclass[11pt,twoside]{article}
\usepackage{multicol}
\usepackage{amsmath}
\usepackage{faktor}
\usepackage{amssymb}
\usepackage[svgnames]{xcolor}
\usepackage{array}
\usepackage{a4wide}
\usepackage{amsthm}

\usepackage{subfigure}
\usepackage{pst-node}
\usepackage[colorlinks=true,linkcolor=blue,citecolor=blue]{hyperref}
\usepackage{mathtools}
\usepackage{tikz-cd} 
\usetikzlibrary{arrows,shapes,positioning}
\usepackage{graphicx}
\usepackage{tikz}
\usepackage{wrapfig}
\usetikzlibrary{arrows}
\usetikzlibrary{decorations.markings}
\usepackage[nottoc]{tocbibind}
\usepackage{cite}
\usepackage{fancyhdr}
\usepackage{graphicx}
\usetikzlibrary{calc,arrows}
\usepackage{mathtools}

\newtheorem{theorem}{Theorem}[section]

\newtheorem{lem}[theorem]{Lemma}

\theoremstyle{definition}
\newtheorem{Def}[theorem]{Definition}

\newtheorem{rem}[theorem]{Remark}

  \newcommand{\properideal}{%
  \mathrel{\ooalign{$\lneq$\cr\raise.22ex\hbox{$\lhd$}\cr}}}
  \tikzstyle directed=[postaction={decorate,decoration={markings,
    mark=at position .4 with {\arrow[very thick, blue] {stealth} }}}]
\tikzstyle reverse directed=[postaction={decorate,decoration={markings,
    mark=at position .8 with {\arrowreversed[very thick, blue]{stealth}}}}]
\numberwithin{equation}{section}

\title{A Hamiltonian approach for point vortices on non-orientable surfaces II: the Klein bottle}
\author{Nataliya A. Balabanova}
\def\defn#1{{\itshape\bfseries#1}}
\newcommand{\Gt}{\widetilde{G}}
\newcommand{\Rt}{\widetilde{R}}

\begin{document}

\maketitle
\begin{abstract}
   This is the second of  two companion papers dedicated to the  investigation of vortex motion on non-orientable surfaces.
   The first paper of the pair is predominantly concerned with establishing the Hamiltonian approach to systems of point vortices on non-orientable manifolds and investigating the limits of the intrinsic (restricted to the non-orientable manifold) approach. In addition, point vortex motion on the M\"obius band is closely examined. In this paper, we investigate dynamics of one and two point vortices on the Klein bottle 
 through establishing explicit forms of the Hamiltonian, equations of motion on the charts, describing relative equilibria, etc. 

 \end{abstract}
   \tableofcontents 
\section{Introduction}

This work is the second of two companion papers dedicated to point vortex motion on non-orientable manifolds. 

In the first paper of the pair \cite{balabanova-montaldi} (henceforth referred to as Part I), we described the intrinsic Hamiltonian  approach to description of motion of point vortices on non-orientable manifolds. 

Since non-orientable manifolds cannot have a symplectic form, the Hamiltonian approach requires passing from the non-orientable manifold $M$ to its orientable double cover $\widetilde{M}$; the phase space for the system consisting of $N$ point vortices will then be $\underbrace{\widetilde{M}\times\ldots\times\widetilde{M}}_{N\  \mathrm{times}}\backslash\widetilde{\Delta}$ with $\widetilde{\Delta}$ the very big diagonal excluding not only n-tuples with repeated points, but also those that have any two points from $\widetilde{M}$ projecting to the same one on $M$.

Perceiving strengths of point vortices as twisted scalars and lifting the system on the non-orientable manifold to the one on its double cover allows us to write the symplectic form on the phase space and the Hamiltonian $\mathcal{H}$. In fact, the latter admits a formulation as a function on  $M\times\ldots\times M$, allowing us to write local equations of motion. The intrinsic form of $\mathcal{H}$ on $M\times\ldots\times M$ is 
\begin{equation} \label{eq: Ham of the double cover}
   \mathcal{H}(x_1,\ldots,x_N) =  -\sum_{i< j}\Gamma_i\Gamma_j\,G_M(x_i,x_j)
   	- \tfrac12 \sum_j\Gamma_j^2R_M(x_j),
\end{equation}
where 
\begin{equation}
\label{eq: twisted Green}
G_M(x,y) = \Gt(x',y') -\Gt(x',\tau(y'))
\end{equation} is the twisted Green's function,
\begin{equation}
\label{eq: twisted Robin}
     R_M(x) = \Rt(x')-\Gt(x',\tau(x'))
\end{equation}
the regular Robin function and $x',\ y'$ the preferred lifts of $x$ and $y$ to $\widetilde{M}$ . 

The momentum map with respect to  the symmetry group $G$  with a Lie algebra $\mathfrak{g}$, originally defined on $\widetilde{M}\times\ldots\widetilde{M}\backslash\widetilde{\Delta}$ as 
$$\Phi(z_1,\dots,z_N) = \sum_j\Gamma_j\,\psi(z_j),$$
with $\psi:\widetilde{M}\to\mathfrak{g}^{\ast}$ the momentum map on one copy of $\widetilde{M}$, also descends to a regular function on $M\times\ldots\times M$. 

Despite the absence of a globally defined Hamiltonian vector field, we may write the equations on each of the charts as the system moves; the trajectories can be glued together making the motion global.  

Recall that a solitary point vortex on an orientable manifold $\widetilde{M}$ of strength $\Gamma$  and with centre at $x_0\in\widetilde{M}$ creates a flow with vorticity $\omega = \Gamma\delta_{x_0}$, where $\delta$ is the Dirac delta-function. 

When $\widetilde{M}$ is compact, this formula has to be amended, so that $\iint\limits_M \omega = 0$ (\cite{dritschel2015motion}). This is achieved through subtracting the inverse of the area $A$ of $\widetilde{M}$: therefore, in the compact case the  vorticity will be $\omega = \Gamma(\delta_{x_0} - 1/A)$. 

Now suppose that the compact manifold $\widetilde{M}$ is a double cover of some compact non-orientable manifold $M$ (e.g. a torus and the Klein bottle).

 Keeping the notation in line with Part I, let $x', \ x'' \in\widetilde{M}$ denote the two preimages of $x\in M$ (if $\tau$ is the antisymplectic involution of $\widetilde{M}$, $\tau(x') = x''$). Then a point vortex of strength $\Gamma$ at the point  $x \in M$ (in the context of the non-orientable manifold, $\Gamma$ is a twisted scalar) is covered by two point vortices of strengths $\Gamma$ and $-\Gamma$, which we consider to be placed respectively at $x' $ and $x''$ (for details on the covering systems and lifts, see Part I, Sections 2.3,2.4).
 
 Then the total vorticity of the fluid on $\widetilde{M}$ will be $\Gamma\delta_{x'} - \Gamma\delta_{x''}$, the integral of which over $\widetilde{M}$ can be easily seen to be 0.

\subsubsection*{The Hamiltonian}
The standard model of the M\"obius band employed in Part I was a strip in the plane with oppositely oriented sides and a cylinder as a double cover. In this instance, there is only one way of constructing the double cover, as well as of simplifying the Hamiltonian from the infinite sum form (see Section 3.1 of Part I and  Section 2 of \cite{montaldi2003vortex} for the constructions). 

As we have remarked above, the Klein bottle is a compact manifold and our go-to model is the $\pi$-by-$\pi$ square in the plane, as in the left part of Figure $\ref{fig:model of the Klein bottle}$. The double cover will always be a torus, but it is straightforward that two ways of constructing it exist: as depicted in Figure \ref{fig:model of the Klein bottle} and in the Figure \ref{fig:model of the Klein bottle alt}. The Hamiltonian for the covering system can be easily seen to have the form of an  infinite double sum.

Below (Section \ref{sec: Ham Klein}), we demonstrate that out of the two natural methods of summing up the Hamiltonian, each one is compatible with only one periodization; changing either of the ingredients leads to a not well-defined function. Namely, the Hamiltonian loses its invariance under $\mu$, the orientation changing involution of the double cover. We establish the form of the Hamiltonian compatible with our natural periodization and demonstrate that it conforms to all the conditions, i.e.\ is a well-defined (regular) function on the Klein bottle.

\subsection*{Invariants and motion of a small number of point vortices}

For computational purposes, we assume that the square that is our model lies in the $x-y$ plane, with $-\frac{\pi}{2}\le x, y\le \frac{\pi}{2}$. In Section \ref{sec: motion one vort} we  describe the motion of a single point vortex, establishing that it moves in a straight line with a constant velocity, unless placed on the lines $y = 0,\pm\frac{\pi}{4},\pm\frac{\pi}{2}$, in which case it remains stationary. 

Systems on the Klein bottle will be $S^1$-symmetric, in the same way as the ones on the M\"obius band; the form of the invariants of motion will coincide as well:
\[
C = \sum_{i=1}^N\Gamma_iy_i
\]
for the system of point vortices with strengths $\Gamma_k$ placed respectively at $(x_k,y_k)$. 

However, compactness of the Klein bottle implies that this is a local invariant of motion, rather than global. Restoring the motion of two point vortices in the spirit of Part I requires us to transition to the covering space in a different way: alternations needed to be done to the method are  the topic of discussion in  Section \ref{sec: two point vortices}.

For the $\pi$-by-$\pi$ square model of the Klein bottle, the cover in question will be a cylinder of radius $2\pi$; on it, we follow the motion of a certain pair of vortices, not caring whether they leave our copy of the bottle. This approach enables us to treat $C$ as a global rather than a local invariant and transfer to the reduced Hamiltonian  while losing no information about the motion of the system. 

We provide numerical evidence of the reduced Hamiltonian having critical points only on the line $x_1-x_2 = 0,\pm\pi$ (a rigorous proof seems to present an insurmountable challenge)
and describe all the possible singularities occurring from collisions between the vortices themselves or their covering copies. Armed with that, we can restore the trajectories of original motion from the reduced system and restore the trajectories.





\section{The Hamiltonian and equations of motion}
\label{sec: the Ham and motion equations}
 We adopt different methods of periodization to figure out the form of the Hamiltonian on the Klein bottle. Our go-to model is the $\pi$-by-$\pi$ square that has a $2\pi$-by-$\pi$ torus as its double cover (Figure \ref{fig:model of the Klein bottle}). In this context, we refer to the oppositely oriented  sides as the \defn{vertical imaginary boundary}  and to the two sides with the same orientation as the  \defn{horizontal imaginary boundary}.

\subsection{Jacobi Theta functions}

Prior to starting the calculations, we need to introduce the functions we will be employing. For details on the topic past the definition and some initial properties, see, for example, \cite{whittaker2020course, bellman2013brief}.

\begin{Def}
\label{def: theta functions}
Jacobi theta functions are the quasi-doubly periodic functions of two complex arguments $z$ and $q$, given by the formulae:
\begin{equation*}
    \begin{split}
        &\theta_1(z,q):=\sum\limits_{n=-\infty}^{+\infty}(-1)^{n-\frac{1}{2}}q^{\left(n+\frac{1}{2}\right)^2}e^{(2n+1)iz} = 2\sum\limits_{n=0}^{+\infty}(-1)^{n}q^{\left(n+\frac{1}{2}\right)^2}\sin((2n+1)z)\\
        &\theta_2(z,q):=\sum\limits_{n=-\infty}^{+\infty}q^{\left(n+\frac{1}{2}\right)^2}e^{(2n+1)iz} = 2\sum\limits_{n=0}^{+\infty}q^{\left(n+\frac{1}{2}\right)^2}\cos((2n+1)z)\\
        &\theta_4(z,q):=\sum\limits_{n=-\infty}^{+\infty}(-1)^{n}q^{n^2}e^{2niz} = 2\sum\limits_{n=0}^{+\infty}(-1)^{n}q^{n^2}\cos(2nz)
    \end{split}
\end{equation*}
The functions can also be expressed in  terms of the number $\tau$, such that  $q = e^{i\pi\tau}$. When $q$ (and, consequently, $\tau$)  is fixed, we will refer to theta functions as $\theta_i(z)$. 
\end{Def}
 Let $\theta'(z,q)$ be the derivative of $\theta(z,q)$ with respect to $z$. It is easy to see that the following relations hold: 
\begin{align}
    \label{eq: Jac th element props}
     &\theta_{1}(-z,q) = -\theta_{1}(z,q),\ &\theta_{2}(-z,q) = \theta_{2}(z,q), \\
    \nonumber &\theta'_{1}(-z,q) = \theta'_{1}(z,q),\ &\theta'_{2}(-z,q) = -\theta'_{2}(z,q),\\
  \nonumber  & \theta_{1,2}(\bar{z},\bar{q}) = \overline{\theta_{1,2}(z,q)},
     \end{align}
    
     as well as some periodic ones (\cite{du1973elliptic}):
     \begin{equation}
     \label{eq: more properties of thetas}
\begin{split}
 &\theta_2(z) = \theta_1(z + \frac{\pi}{2}), \ \theta_1(z+\pi) = -\theta_1(z),\
         \theta_1'(z + \pi) = -\theta_1'(z),\
         \theta_1''(z + \pi) = -\theta_1''(z),\ \\
&\frac{\theta'_1(z + \pi)}{\theta_1(z + \pi)} = \frac{\theta'_1(z)}{\theta_1(z)},\
          \frac{\theta'_1(z + \pi\tau)}{\theta_1(z + \pi\tau)} = -2i +\frac{\theta'_1(z)}{\theta_1(z)}, \theta_1(z + \pi\tau)= -\frac{e^{-2 i z}}{q}\theta_1(z).
\end{split}
\end{equation}
Armed with this, we proceed with the explicit computation for the Hamiltonian.

\begin{figure}
    \centering
    \begin{tikzpicture}[scale=0.3]
    \draw[reverse directed] (-5,-4) -- (-5,8);
    \draw[directed] (-5,8) -- (7,8);
    \draw[directed] (7,8) -- (19,8);
     \draw[directed] (-5,-4) -- (7,-4);
      \draw[directed] (7,-4) -- (19,-4);
     \draw[ directed] (7,-4) -- (7,8);
      \draw[dashed] (-5,2) -- (19,2);
       \draw[ reverse directed] (19,-4) -- (19,8);
       \draw  (-2,0) circle (1.5);
\draw[decorate,decoration={ markings,mark=at position -3 cm with
{\arrow[line width= 0.3mm]{>}};}]{ (-2,0) circle (1.5)};
\draw[fill] (-2,0 ) circle (0.05);
\draw  (10,4) circle (1.5);
\draw[decorate,decoration={ markings,mark=at position -3 cm with
{\arrow[line width= 0.3mm]{<}};}]{ (10,4) circle (1.5)};
\draw[fill] (10,4 ) circle (0.05);
\end{tikzpicture}
    \caption{A $2\pi$-by-$\pi$ torus as a double  cover of a $\pi$-by-$\pi$ Klein bottle. The dashed line is $y = 0$}
     \label{fig:model of the Klein bottle}
\end{figure}
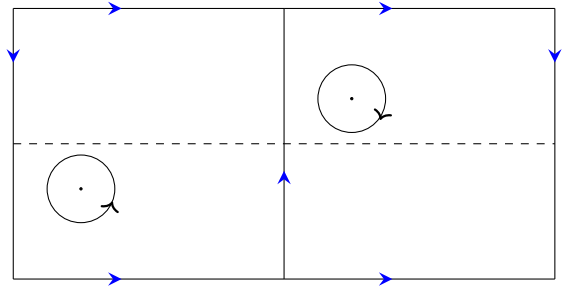

\subsection{The Hamiltonian for the torus}

It is straightforward that the Hamiltonian for the Klein bottle can be obtained from the one for the torus through imposing certain symmetries on the system; \cite{o1989hamiltonian} gives the explicit form of the Hamiltonian for the $\pi$-by-$\tau\pi$ torus: 
 \begin{equation}
\label{eq: Ham for torus}
\mathcal{H}_T =  -\frac{1}{2\pi}\sum_{k<l}\Gamma_k\Gamma_l\left(\log\left|\theta_1\left(z'_k-z'_l, e^{i\pi\tau}\right)\right| - \frac{\left(\mathrm{Im}(z'_k-z'_l)\right)^2}{\pi\mathrm{Im}\tau}\right).
\end{equation}
The choice of the function is justified by the location of its poles as well as its periodicities. However, in the scope of this work, we aim to provide an explicit calculation. 

Consider a lattice comprised of $\pi$-by-$2\pi$ rectangles (as in Figure \ref{fig: periodization}) and an appropriately periodized system of point vortices. Then the Hamiltonian is given by an infinite (divergent) double sum:
\begin{equation}
\label{eq: Hamiltonian tirus double sum}
\mathcal{H} = -\frac{1}{2\pi}\sum_{k<l}\sum_{m,n}\Gamma_k\Gamma_l\log\left|z_k-z_l +\pi i n + 2\pi m \right|
\end{equation}
There are two natural ways of computing  $\sum\limits_{m,n}\log\left|z_k-z_l +\pi i n + 2\pi m \right|$:  summing up horizontally and then vertically and the other way around. Physical reasons dictate that the two answers coincide (up to addition of a constant). However, as we will see below, the form of the two is  drastically different.

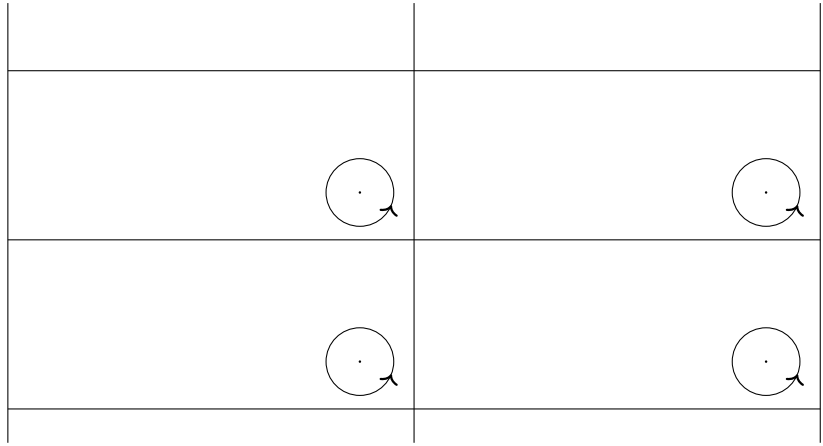
\begin{figure}
\centering
 \begin{tikzpicture}[scale = 0.9]
    \draw[] (-5,-1.5) -- (-5,5);
         \draw[] (1,-1.5) -- (1,5);
     \draw[] (7,-1.5) -- (7,5);
      \draw[] (-5,4) -- (7,4);
           \draw[] (-5,1.5) -- (7,1.5);
             \draw[] (-5,-1) -- (7,-1);
\draw  (0.2,-0.3) circle (0.5);
\draw[decorate,decoration={ markings,mark=at position -3 cm with
{\arrow[line width= 0.3mm]{>}};}]{ (0.2,-0.3) circle (0.5)};
\draw  (0.2,2.2 ) circle (0.5);
\draw[decorate,decoration={ markings,mark=at position -3 cm with
{\arrow[line width= 0.3mm]{>}};}]{ (0.2,2.2 ) circle (0.5)};
\draw  (6.2,-0.3 ) circle (0.5);
\draw[decorate,decoration={ markings,mark=at position -3 cm with
{\arrow[line width= 0.3mm]{>}};}]{ (6.2,-0.3 ) circle (0.5)};
\draw  (6.2,2.2) circle (0.5);
\draw[decorate,decoration={ markings,mark=at position -3 cm with
{\arrow[line width= 0.3mm]{>}};}]{ (6.2,2.2) circle (0.5)};
\draw[fill] (0.2,-0.3 ) circle (0.01);
\draw[fill] (0.2,2.2 ) circle (0.01);
\draw[fill] (6.2,-0.3 ) circle (0.01);
\draw[fill] (6.2,2.2 ) circle (0.01);
\end{tikzpicture}
   \caption{Double periodization for a single vortex on the torus}
    \label{fig: periodization}
\end{figure}

We give detailed computations for summing up horizontally then vertically. To force our infinite sum to converge, we employ the same trick as \cite{montaldi2003vortex}: subtracting an infinitely large constant number from our Hamiltonian. This transition will be denoted by $\rightarrow$ in the calculations below. 
\begin{small}
\label{eq: long calculation}
\begin{equation*}
    \begin{split}
      \sum_{m,n}\log\left|z_k-z_l +\pi i n + 2\pi m \right|  =& \log\left|\prod_{ m>0,n}\left((z_k-z_l +\pi i n)^2 - 4\pi^2 m^2\right) \ast\prod_n\left(z_k-z_l +\pi i n\right)\right|\xrightarrow[]{A} \\
      \xrightarrow[]{A}&\log\left|\prod_n\sin\left(\frac{z_k-z_l + \pi n i}{2}\right)\right|= \\=&
      \log\left|\prod_n\left(\sin\left(\frac{z_k-z_l }{2}\right)\cosh\left(\frac{\pi n }{2}\right) +i\cos\left(\frac{z_k-z_l }{2}\right)\sinh\left(\frac{\pi n }{2}\right)\right)\right| = \\=&
    \log\left|\prod_{n>0}\left(\sin^2\left(\frac{z_k-z_l }{2}\right)\cosh^2\left(\frac{\pi n }{2}\right) +\cos^2\left(\frac{z_k-z_l }{2}\right)\sinh^2\left(\frac{\pi n i}{2}\right)\right)\right| + \\ +& \log\left|\sin\left(\frac{z_k-z_l}{2}\right)\right| =\\ =&\log\left|\prod_{n>0}\left(\cos^2\left(\frac{z_k-z_l}{2}\right) - \cosh^2\left(\frac{\pi n}{2}\right)\right)\right| +\log\left|\sin\left(\frac{z_k-z_l}{2}\right)\right|\xrightarrow[]{B}\\\xrightarrow[]{B}&\log\left|\theta_1\left(\frac{z_k-z_l}{2},e^{-\frac{\pi}{2}}\right)\right|.
    \end{split}
\end{equation*}
\end{small}
Here we used that $\sin(ix) = i\sinh(x), \ \cos(ix)  = \cosh(x)$ for $x\in \mathbb{R}$, as well as the representations of theta functions through infinite products from  \cite{bonn}. In the transition $\xrightarrow[]{A}$ we subtract  $\sum\limits_{m>0}\log|4\pi^2m^2|$ from our sum, and in $\xrightarrow[]{B}$ we subtract $\sum\limits_{n>0}\log\left|\cosh^2\left(\frac{\pi n }{2} \right)\right|$.

The function  $\log\left|\theta_1\left(\frac{z_k-z_l}{2},e^{-\frac{\pi}{2}}\right)\right|$ is $2\pi$-periodic in real parts of $z_j$; however, it is quasi $\pi$-periodic in imaginary and hence, not well-defined on the torus. The initial periodicity of (\ref{eq: Hamiltonian tirus double sum}) was lost when we `    folded' the sum, assuming that $m$ and $n$ are greater than 0. 

In order to make the last expression in (\ref{eq: long calculation}) periodic, we need to subtract $\frac{\left(\mathrm{Im}\left(z_k-z_l\right)\right)^2}{2\pi}$ (the periodicity of the resulting function can be checked from the last relation on $\theta_1$ in (\ref{eq: more properties of thetas})). Therefore, the Hamiltonian on the torus obtained from this periodizing is given by 
\begin{equation}
    \label{eq: Ham torus horizontal periodizing}
   \mathcal{H}^T_1 =  -\frac{1}{2\pi}\sum_{k<l}\Gamma_k\Gamma_l\left(\log\left|\theta_1\left(\frac{z_k-z_l}{2},e^{-\frac{\pi}{2}}\right)\right| - \frac{\left(\mathrm{Im}\left(z_k-z_l\right)\right)^2}{2\pi}\right). 
\end{equation}
\begin{rem}
 Observe that in (\ref{eq: Ham torus horizontal periodizing}) $\log\left|\theta_1\left(\frac{z_k-z_l}{2},e^{-\frac{\pi}{2}}\right)\right|$ is the part of the Green function obtained through the method of images, while $\frac{\left(\mathrm{Im}\left(z_k-z_l\right)\right)^2}{4\pi^2}$ after the application of $\Delta$ yields $1/2\pi^2$, which is the inverse of the surface area of the torus, in accordance with the remark on the strength of point vortices on closed orientable surfaces above.
\end{rem}
Now, we sum (\ref{eq: Hamiltonian tirus double sum}) vertically then horizontally, to obtain:
\begin{small}
\begin{equation*}
    \begin{split}
      \sum_{m,n}\log\left|z_k-z_l +\pi i n + 2\pi m \right|  =& \log\left|\prod_{n>0, m}\left((z_k-z_l +2 \pi m)^2 + \pi^2 n^2\right) \ast\prod_m\left(z_k-z_l +2\pi m\right)\right| \xrightarrow[]{A'} \\
      \xrightarrow[]{A'}&\log\left|\prod_n\sinh\left(z_k-z_l + 2\pi m\right)\right|= \\=&
      \log\left|\prod_n\left(\sinh\left(z_k-z_l\right)\cosh\left(2\pi m\right) +\cosh\left(z_k-z_l\right)\sinh\left(2\pi m\right)\right)\right| \xrightarrow[]{B'}\\\xrightarrow[]{B'}&\log\left|\theta_1\left(i\left(z_k-z_l\right),e^{-2\pi}\right)\right|.
    \end{split}
\end{equation*}
\end{small}

Transition $\xrightarrow[]{A'}$ is subtraction of $\sum\limits_{n>0}\log\left|\pi^2n^2\right|$ and $\xrightarrow[]{B'}$  of $\sum\limits_{m>0}\log\left|\sinh^2\left(2\pi m\right)\right|$.

This function is $\pi$-periodic in imaginary parts and $2\pi$-quasi periodic in real. Analogously, we remedy that through subtracting $\frac{\left(\mathrm{Re}(z_k-z_l)\right)^2}{2\pi}$, to obtain
\begin{equation}
    \label{eq: Ham torus vertical}
    \mathcal{H}^T_2=  -\frac{1}{2\pi}\sum_{k<l}\Gamma_k\Gamma_l\left(\log\left|\theta_1\left(i\left(z_k-z_l\right),e^{-2\pi}\right)\right| - \frac{\left(\mathrm{Re}(z_k-z_l)\right)^2}{2\pi}\right).
\end{equation}
\begin{rem}
 Observe how the change $z\mapsto iz$ transforms (\ref{eq: Ham torus vertical}) into (\ref{eq: Ham for torus}), as written for a $\pi$-by-$2\pi$ torus: indeed, multiplication by $i$ is the ninety degree rotation of the plane.
\end{rem}
We have remarked above that the functions $\mathcal{H}^T_1$ and $\mathcal{H}^T_2$ must differ by a constant; however, they look nothing like each other. Nonetheless, the following holds:

\begin{lem}
$\mathcal{H}_1 + \frac{\log(2)}{4\pi}\sum\limits_{k<l}\Gamma_k\Gamma_l  = \mathcal{H}_2$.
\end{lem}
\begin{proof}
We employ the equality from \cite{bonn}:
\begin{equation}
\label{eq: connection with i}
\frac{1}{\sqrt{\lambda}}\,e^{z^2/(\pi\lambda)}\,\theta_1\left(\lambda^{-1}z, e^{-{\pi}/{\lambda}}\right)= -i\theta_1\left(iz, e^{-\pi\lambda}\right)
\end{equation}
in order to write for each $z_k$- $z_l$ pair:
\begin{small}
\begin{equation*}
    \begin{split}
     & \log\left|\theta_1\left(\frac{z_k-z_l}{2},e^{-\frac{\pi}{2}}\right)\right|   - \frac{(\mathrm{Im}(z_k-z_l))^2}{2\pi} =\\&=\log\left|\theta_1\left(i(z_k-z_l),e^{-2\pi}\right)\right|   - \frac{(\mathrm{Im}(z_k-z_l))^2}{2\pi}- \log\left|\mathrm{exp}\left(\frac{(z_k-z_l)^2}{2\pi}\right)\right| + \log(\sqrt{2}) = \\&=\log\left|\theta_1\left(i(z_k-z_l),e^{-2\pi}\right)\right|   - \frac{(\mathrm{Im}(z_k-z_l))^2}{2\pi} - \frac{(\mathrm{Re}(z_k-z_l))^2}{2\pi} + \frac{(\mathrm{Im}(z_k-z_l))^2}{2\pi} + \log(\sqrt{2})=\\&=\log\left|\theta_1\left(i(z_k-z_l),e^{-2\pi}\right)\right|  - \frac{(\mathrm{Re}(z_k-z_l))^2}{2\pi} + \log(\sqrt{2}),
    \end{split}
\end{equation*}
\end{small}
whence the statement of the lemma follows immediately.
\end{proof}
\subsection{ The Hamiltonian for the Klein bottle}
\label{sec: Ham Klein}
The next step is to periodize the Hamiltonian for the torus in order to get one for the Klein bottle. Since $\tau$ in this context denotes the parameter in Jacobi functions $\theta_j$, we use $\mu$ for the antisymplectic involution  on the double cover.

Our rectangle covers two copies of the Klein bottle, as drawn in the Figure \ref{fig:model of the Klein bottle}. If we suppose that the bottom left corner of the rectangle is placed at the point with coordinates $(-\pi/2, -\pi/2)$, the involution will be explicitly given by $\mu:
z\mapsto\bar{z} + \pi$.

We periodize both forms of the Hamiltonian to see if two functions that are seemingly so different will give the same results.  Consider first $\mathcal{H}^T_2$.

Observing that $\left|\theta_1(z,q)\right| = \left|\theta_1(\bar{z},q)\right|$ when $q\in\mathbb{R}$ allows us to simplify it to the expression (we also subtract a constant and divide the final expression by 2, as per the discussion  in Section 2.4 of  Part I ): 
\begin{equation}
    \begin{split}
    \label{eq: Ham the wromg one}
        \mathcal{H}_0 =& -\frac{1}{2\pi}\sum_{k<l}\Gamma_k\Gamma_l\log\left|\theta_1\left(i(z_k-z_l), e^{-2\pi}\right)\right| + \frac{1}{4\pi}\sum_{k\ne l}\Gamma_k\Gamma_l\log\left|\theta_1\left(i(z_k-\bar{z}_l- \pi), e^{-2\pi}\right)\right| + \\&+\frac{1}{4\pi}\sum_k\Gamma_k^2\log\left|\theta_1\left(2y_k,e^{-2\pi}\right)\right|
    \end{split}
\end{equation}

The function (\ref{eq: Ham the wromg one}) has the same periodicity properties as its counterpart on the torus; however, is it invariant under $\mu$?
\begin{lem}
$\mathcal{H}_0(\mu(z_1), z_2,\ldots,z_n) = \mathcal{H}_0(z_1,\ldots,z_n) - \frac12\sum\limits_{k\ne 1}\Gamma_1\Gamma_k$.
\end{lem}
\begin{proof}
From \cite{whittaker2020course}, we know that the following relation holds for $z\in\mathbb{C}$:
\[
\frac{\theta_1(z + 2i\pi, e^{-2\pi})}{\theta_1(z, e^{-2\pi})} = -e^{2\pi -2iz}.
\]
Considering that the involution changes the sign of $\Gamma_1$, a computation yields 
\begin{equation*}
    \begin{split}
        \mathcal{H}_0(\mu(z_1), z_2,\ldots,z_n) &= \mathcal{H}_0(z_1,\ldots,z_n) + \frac{1}{4\pi}\sum_{k\ne 1}\Gamma_1\Gamma_k\log\left|\mathrm{exp}\left(2\pi + 2(z_1 - \bar{z}_k - \pi)\right)\right| -\\&- \frac{1}{4\pi}\sum_{k\ne 1}\Gamma_1\Gamma_k\log\left|\mathrm{exp}\left(2\pi + 2(z_1-z_k)\right)\right|= \\&=\mathcal{H}_0(z_1,\ldots,z_n) - \frac{1}{4\pi}\sum_{k\ne 1}\Gamma_1\Gamma_k\log\left|\mathrm{exp}\left(2(\bar{z}_k-z_k) + 2\pi\right)\right| =\\&= \mathcal{H}_0(z_1,\ldots,z_n)- \frac{1}{2}\sum_{k\ne 1}\Gamma_1\Gamma_k
    \end{split}
\end{equation*}
\end{proof}

Therefore, this function is not well-defined on the square model of the Klein bottle: this form of the Hamiltonian on the torus is incompatible with this concrete periodization. 

However, a different method of periodization yields a well-defined Hamiltonian: suppose our double cover is as in Figure \ref{fig:model of the Klein bottle alt}, i.e. the Klein bottle is $2\pi$-by-$\pi/2$. Then $\mu':z\mapsto -\bar{z} + i\frac{\pi}{2}$. It can be checked that this periodization, as applied to $\mathcal{H}^T_2$, gives a Hamiltonian with proper periodicities and invariant under $\mu'$. 
\begin{figure}
    \centering
    \begin{tikzpicture}[scale=0.3]
    \draw[ directed] (-5,-4) -- (-5,8);
    \draw[directed] (-5,8) -- (19,8);
      \draw[directed] (-5,-4) -- (19,-4);
     \draw[ directed] (7,-4) -- (7,8);
      \draw[reverse directed] (-5,2) -- (19,2);
       \draw[directed] (19,-4) -- (19,8);
       \draw  (4,4) circle (1.5);
\draw[decorate,decoration={ markings,mark=at position -3 cm with
{\arrow[line width= 0.3mm]{>}};}]{ (4,4) circle (1.5)};
\draw[fill] (4,4 ) circle (0.05);
\draw  (10,-2) circle (1.5);
\draw[decorate,decoration={ markings,mark=at position -3 cm with
{\arrow[line width= 0.3mm]{<}};}]{ (10,-2) circle (1.5)};
\draw[fill] (10,-2 ) circle (0.05);
\end{tikzpicture}
    \caption{A $2\pi$-by-$\pi$ torus as a double  cover of a $2\pi$-by-$\pi/2$ Klein bottle}
     \label{fig:model of the Klein bottle alt}
\end{figure}
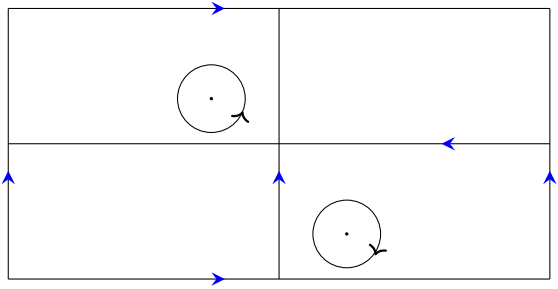

The square model is more convenient for  computational purposes, so we  obtain the final Hamiltonian  for the Klein bottle from periodizing (and again, dividing by 2) $\mathcal{H}^T_1$: 

\begin{small}
\begin{equation}
\begin{split}
    \label{eq: ONE TRUE KLEIN HAM}
    \mathcal{H} =& -\frac{1}{2\pi}\sum\limits_{\alpha<\beta}\Gamma_{\alpha}\Gamma_{\beta}\log\left|\theta_1\left(\frac{z_{\alpha}-z_{\beta}}{2},e^{-\frac{\pi}{2}}\right)\right| + \frac{1}{2\pi}\sum\limits_{\alpha<\beta}\Gamma_{\alpha}\Gamma_{\beta}\log\left|\theta_2\left(\frac{z_{\alpha}-\bar{z}_{\beta}}{2},e^{-\frac{\pi}{2}}\right)\right| + \\&+\frac{1}{2\pi}\sum\limits_{\alpha<\beta}\Gamma_{\alpha}\Gamma_{\beta}\left(\frac{(y_{\alpha} - y_{\beta})^2}{2\pi} - \frac{(y_{\alpha} + y_{\beta})^2}{2\pi}\right) + \frac{1}{4\pi}\sum\limits_{\alpha}\Gamma_{\alpha}^2\left(\log\left|\theta_1\left(iy_{\alpha} - \frac{\pi}{2},e^{-\frac{\pi}{2}}\right)\right| - \frac{2y_{\alpha}^2}{\pi}\right) = \\=&
    -\frac{1}{2\pi}\sum\limits_{\alpha<\beta}\Gamma_{\alpha}\Gamma_{\beta}\log\left|\theta_1\left(\frac{z_{\alpha}-z_{\beta}}{2},e^{-\frac{\pi}{2}}\right)\right| + \frac{1}{2\pi}\sum\limits_{\alpha<\beta}\Gamma_{\alpha}\Gamma_{\beta}\log\left|\theta_2\left(\frac{z_{\alpha}-\bar{z}_{\beta}}{2},e^{-\frac{\pi}{2}}\right)\right| -\\&-\frac{1}{\pi^2}\sum\limits_{\alpha<\beta}\Gamma_{\alpha}\Gamma_{\beta}y_{\alpha}y_{\beta} + \frac{1}{4\pi}\sum\limits_{\alpha}\Gamma_{\alpha}^2\left(\log\left|\theta_1\left(iy_{\alpha} - \frac{\pi}{2},e^{-\frac{\pi}{2}}\right)\right| - \frac{2y_{\alpha}^2}{\pi}\right).     \end{split}
\end{equation}
\end{small}

\begin{lem}
The Hamiltonian (\ref{eq: ONE TRUE KLEIN HAM}) is $\pi$-periodic vertically, $2\pi$-periodic horizontally and invariant under $\mu$. 
\end{lem}

\begin{proof}
The first two statements for our Hamiltonian follow from the corresponding properties of its predecessor on the torus (\ref{eq: Ham torus horizontal periodizing}). The last statement can be easily checked as well: under the change $x_i\mapsto x_i + \pi,\ y_i\mapsto-y_i,\ \Gamma_i\mapsto-\Gamma_i$ the first two summands exchange places, and the last two remain unchanged. 
\end{proof}
\begin{rem}
The Robin function $\log\left|\theta_1\left(iy - \frac{\pi}{2}, e^{-\frac{\pi}{2}}\right)\right| - \frac{2y^2}{\pi}$ does not look like a periodic function; however, the following holds:
\[
\log\left|\theta_1\left(iy - \frac{\pi}{2}, e^{-\frac{\pi}{2}}\right)\right| - \frac{2y^2}{\pi} = \log\left|\theta_4\left(2y,e^{-2\pi}\right)\right| + C.
\]
for constant $C$. 

Indeed, from (\ref{eq: connection with i}), we have
\begin{small}
\begin{equation*}
    \begin{split}
\log\left|\theta_1\left(iy-\frac{\pi}{2},e^{-\frac{\pi}{2}}\right)\right| - \frac{2y^2}{\pi} =&\log\left|\theta_1\left(2y +i\pi,e^{-2\pi} \right)\right| + \log\left|\mathrm{exp}\left(\frac{(2y + i\pi)^2}{2\pi}\right)\right| -\log(\sqrt{2}) - \frac{2y^2}{\pi}=\\=&\log\left|\theta_1\left(2y + i\pi,e^{-2\pi}\right)\right| + C',
\end{split} 
\end{equation*}
\end{small}
where $C'$ is a constant.

On the other hand, we know from \cite{wolfram} that 
\begin{equation*}
    \label{eq: quasiperiodicity}
\theta_1\left(z,q\right) = -ie^{iz + \pi i\tau/4} \ \theta_4\left(z + \frac{1}{2}\pi\tau,q\right),
\end{equation*}
which gives us 
\begin{equation*}
    \begin{split}
        \log\left|\theta_1\left(2y + i\pi,e^{-2\pi}\right)\right| &= \log\left|\theta_4\left(2y + i\pi -i\pi,e^{-2\pi} \right)\right|+ \log\left|\mathrm{exp}\left(2iy- \pi -\frac{\pi}{2}\right)\right|  = \\&=\log\left|\theta_4\left(2y,e^{-2\pi}\right)\right| + C''.
    \end{split}
\end{equation*}
for some constant $C''$.
\end{rem}

Therefore, the Hamiltonian (\ref{eq: ONE TRUE KLEIN HAM}) is the only one with the two required properties and is well-defined on the square model of the  Klein bottle. Hence, this is the energy function we will use going forward. 
\begin{rem}
In the notations of Part I, the Green's function on $M$ is
\begin{small}
\[
G_M(z,w) = \frac{1}{2\pi}\log\left|\theta_1\left(\frac{z-w}{2},e^{-\frac{\pi}{2}}\right)\right|  - \frac{1}{2\pi}\log\left|\theta_2\left(\frac{z-\bar{w}}{2},e^{-\frac{\pi}{2}}\right)\right|  +\frac{1}{\pi^2}\mathrm{Im}(z)\mathrm{Im}(w)
\]
and  the Robin function is
\[
R_M(y) =  \frac{1}{2\pi}\left(\log\left|\theta_1\left(iy - \frac{\pi}{2},e^{-\frac{\pi}{2}}\right)\right| - \frac{2y^2}{\pi}\right).
\]
\end{small}
\end{rem}
\section{Motion of one vortex}
\label{sec: motion one vort}
\begin{figure}
    \centering
   \subfigure[The Robin function]{\includegraphics[scale =0.45]{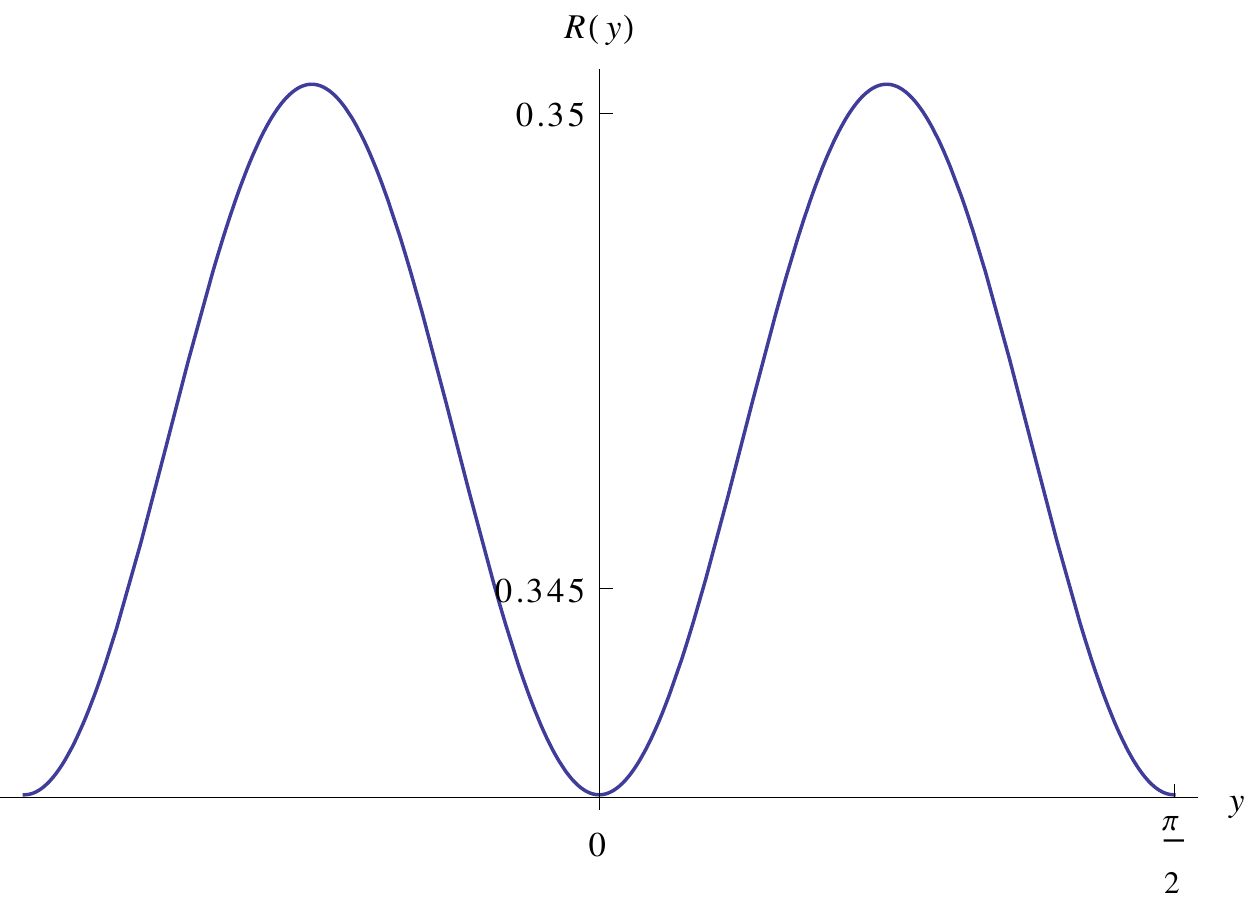}}
   \subfigure[The derivative of the Robin function]{\includegraphics[scale =0.45]{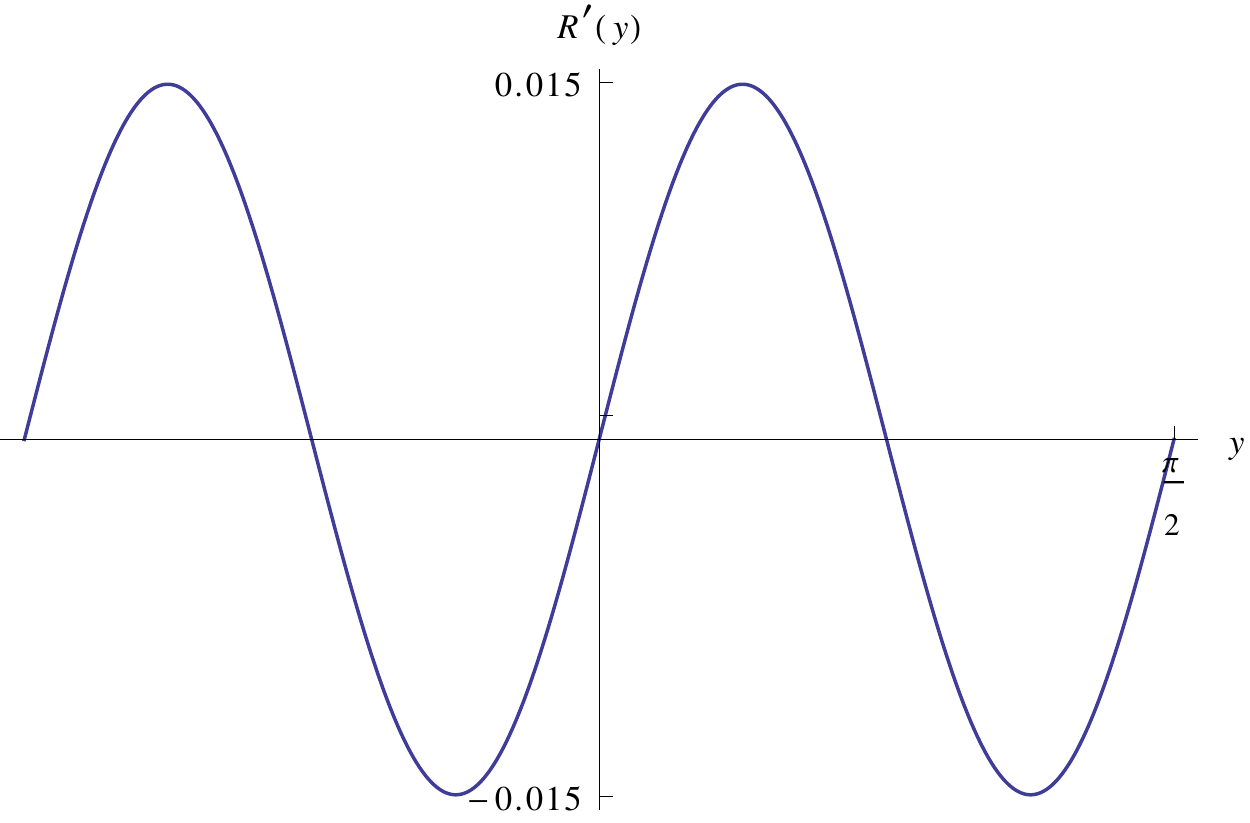}}
   \caption{The Robin function and its derivative}
   \label{fig: Robin function klein}
\end{figure}

A simple calculation gives that the equation of motion of one point vortex of strength $\Gamma$ has the form
\[
\begin{cases}
\dot{x} =  \frac{\Gamma}{4\pi}\left(i\frac{\theta'_1\left(iy - \frac{\pi}{2}\right)}{\theta_1\left(iy - \frac{\pi}{2}\right)} - \frac{4y}{\pi}\right),\\
\dot{y} = 0.
\end{cases}
\]
\begin{rem}
 The expression $\frac{\theta'_1\left(iy - \frac{\pi}{2}\right)}{\theta_1\left(iy - \frac{\pi}{2}\right)}$ is a purely imaginary number when $y\in\mathbb{R}$.
 On the webpage \cite{wolfram} one can find the following relation:
 \begin{equation}
 \label{eq: thetaprime/theta}
 \frac{\theta'_1(z)}{\theta_1(z)} = \cot(z) + 4\sum_{n=1}^{\infty}\frac{q^{2n}\sin(2z)}{q^{4n} - 2q^{2n}\cos(2z) + 1}
 \end{equation}
 When $z = iy - \frac{\pi}{2}$, (\ref{eq: thetaprime/theta}) turns into
 \begin{equation*}
     \begin{split}
         &\cot\left(iy - \frac{\pi}{2}\right)+ 4\sum_{n=1}^{\infty}\frac{q^{2n}\sin(2iy - \pi)}{q^{4n} - 2q^{2n}\cos(2iy - \pi) + 1} = \\&=-i\tanh(y) - 4i\sum_{n=1}^{\infty}\frac{q^{2n}\sinh(2y)}{q^{4n} + 2q^{2n}\cosh(2y ) + 1}
     \end{split}
 \end{equation*}
\end{rem}
The plot of the Robin function  and its derivative are depicted in Figure \ref{fig: Robin function klein} (a) and (b). 

Observe how the Robin function $R(y)$ is even and $\frac{\pi}{2}$-periodic in $y$. Note, however, that the function is not symmetric with respect to the reflection across the $x$-axis and translation by $\pi/4 $ (and neither is, consequently, its derivative). 

The velocity of the point vortex, in turn, is a $\frac{\pi}{2}$-periodic odd function; a solitary  vortex will be stationary if and only if placed on the lines $y = 0,\pm\frac{\pi}{4}, \pm\frac{\pi}{2}$ . Observe that the first and the last lines are the loci of fixed points for the orientation changing isometry, and therefore they will necessarily  be critical points of the Robin function.

\begin{figure}
\centering
 \subfigure[One copy of the bottle]{\includegraphics[scale = 0.42]{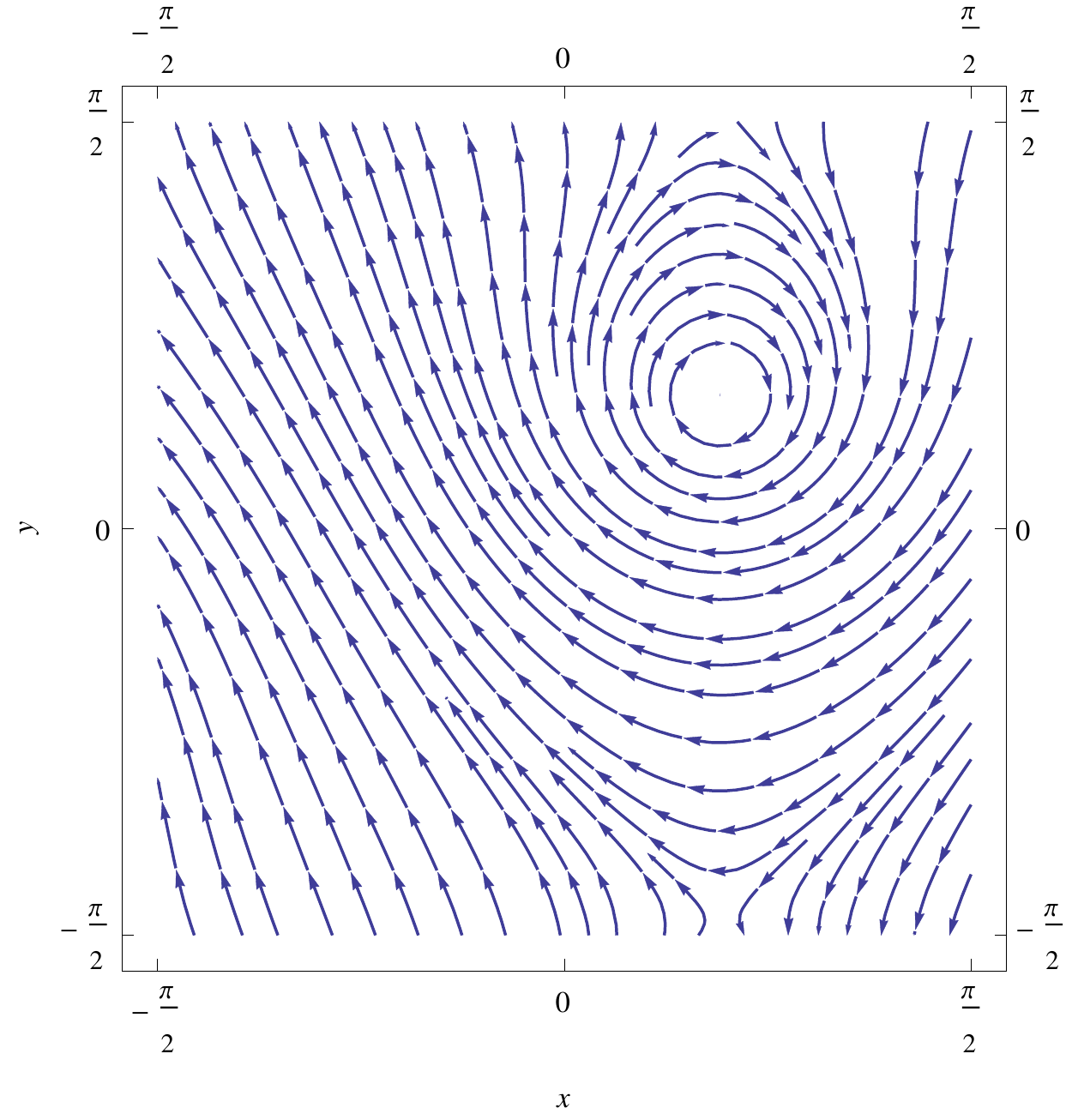}}
\subfigure[Four copies of the bottle]{\includegraphics[scale = 0.42]{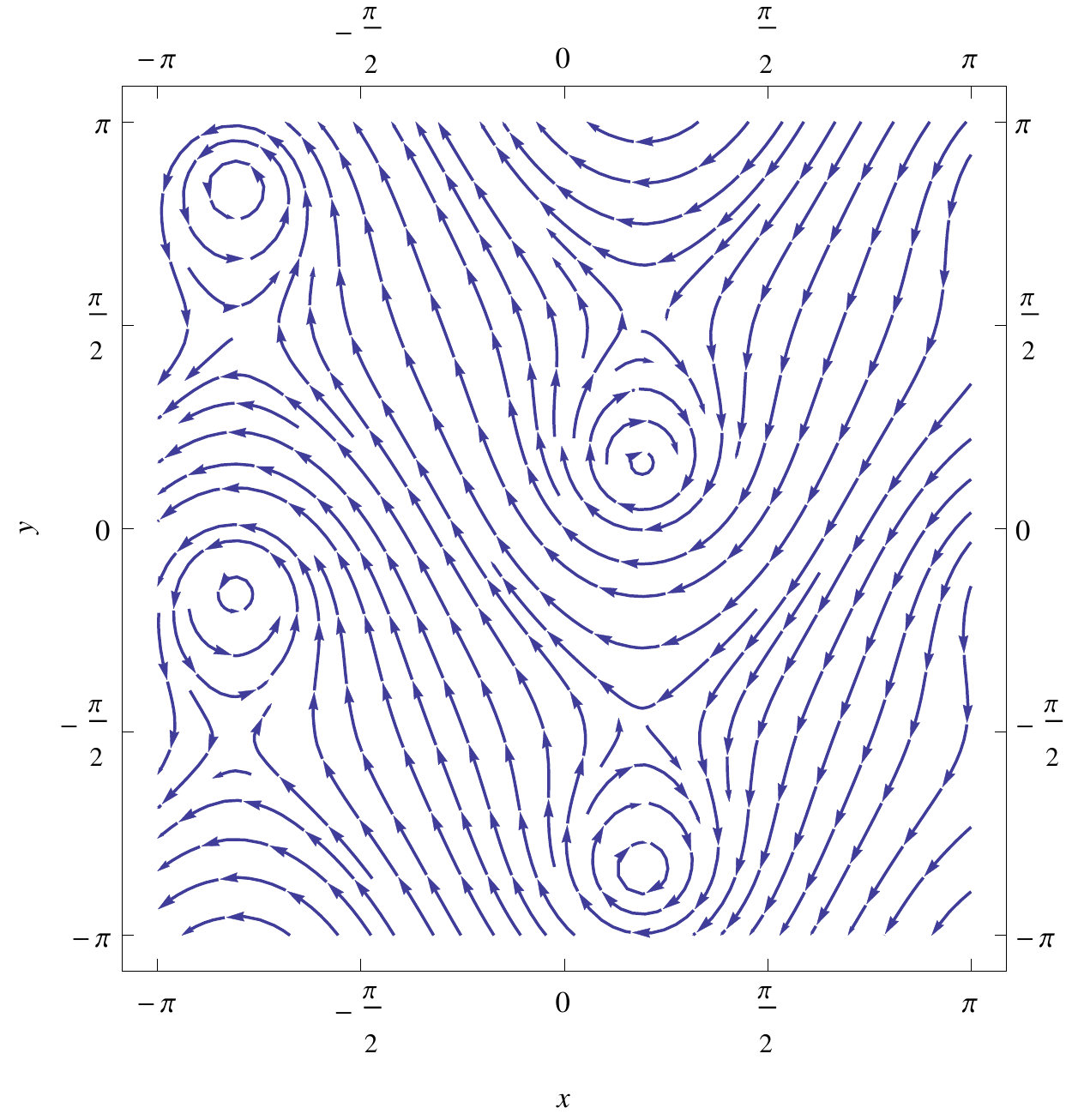}}
 \caption{Vector field created by a single vortex on the Klein bottle}
 \label{fig: vector field single vortex Klein bottle}
 \end{figure}

With the help of the equation of motion (see details below) we can reconstruct the vector field created by one vortex: see Figure \ref{fig: vector field single vortex Klein bottle} (a) for one copy of the band and (b) for multiple.

  \section{Symmetries and invariants}
As in the case of the M\"obius band, the group of symmetries will be $S^1$, acting by horizontal translations. Observe that  now the function $C := \sum\limits_{k}\Gamma_{k}y_{k}$ is  a local invariant of motion.

Due to the nature of the $S^1$-action, relative equilibria on the Klein bottle will behave  identically  to the ones on the M\"obius band: vortices will move horizontally, maintaining a rigid configuration. 

Analogously to Part I and \cite{montaldi2000relative}, we can observe that the following configurations will be fixed and relative equilibria respectively:  
\begin{itemize}
    \item an odd number of  point vortices  on the lines $y = 0$ or $y = \frac{\pi}{2}$ with alternating signs of strengths;  existence of fixed equilibria for arrangements like this can be demonstrated  in precisely the same manner as the one in \cite{montaldi2003vortex}: existence of critical points of he Hamiltonian is deduced from its behaviour at configurations where the vortices collide . 
    \item configurations inherited from $N$-rings on the torus; two aligned rings  when $N$ is even and two staggered ones when $N$ is odd (see \cite{montaldi2000relative} and \cite{laurent2001point} for the application of the Principle of Symmetric Criticality in this case). 
\end{itemize}

  \section{Two vortices}
  \label{sec: two point vortices}
 The Hamiltonian has the form 
 \begin{equation}
 \label{eq: Ham Klein two vort}
 \begin{split}
     \mathcal{H} =& -\frac{1}{2\pi}\Gamma_1\Gamma_2\log\left|\theta_1\left(\frac{z_1-z_2}{2},e^{-\frac{\pi}{2}}\right)\right| + \frac{1}{2\pi}\Gamma_1\Gamma_2\log\left|\theta_2\left(\frac{z_1-\bar{z}_2}{2},e^{-\frac{\pi}{2}}\right)\right| +\\&- \frac{1}{\pi^2}\Gamma_1\Gamma_2y_1y_2 + \frac{1}{4\pi}\Gamma_1^2\left(\log\left|\theta_1\left(iy_1 - \frac{\pi}{2},e^{-\frac{\pi}{2}}\right)\right| - \frac{2y_1^2}{\pi}\right) + \\&+ \frac{1}{4\pi}\Gamma_2^2\left(\log\left|\theta_1\left(iy_2 - \frac{\pi}{2},e^{-\frac{\pi}{2}}\right)\right| - \frac{2y_2^2}{\pi}\right)
     \end{split}
 \end{equation}
  
\begin{figure}
    \centering
    \includegraphics[scale = 0.5]{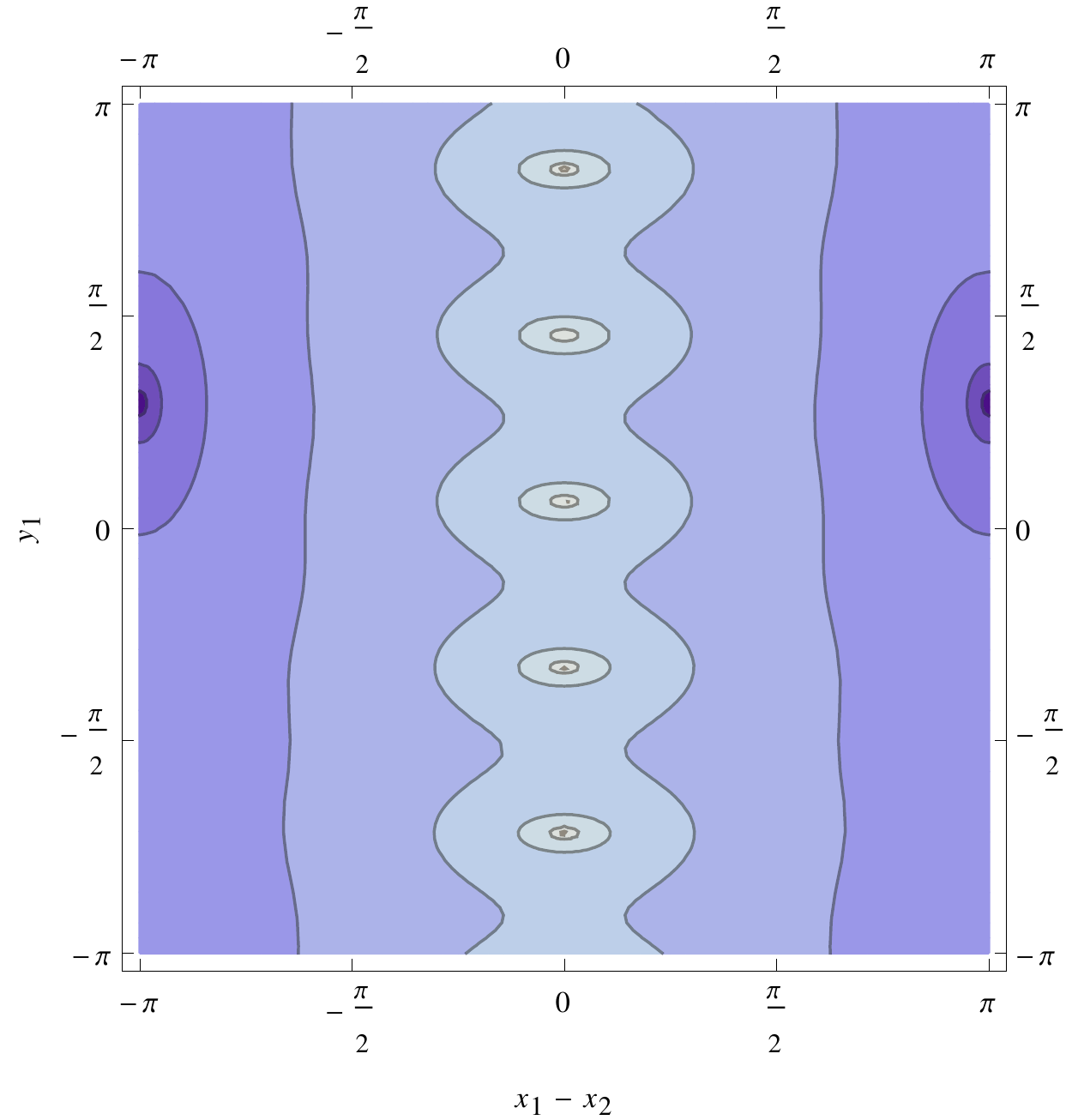}
    \caption{Level sets of the reduced  Hamiltonian for the Klein bottle}
    \label{fig:my_label}
\end{figure}

We recall the following formula form \cite{montaldi2003vortex}:
\[
\dot{z}_k = -2i\frac{\partial\mathcal{H}}{\partial \bar{z}_k}.
\]
Rewriting (\ref{eq: Jac th element props}) through $\bar{z}_k$ with the help of the relations (\ref{eq: more properties of thetas}), we get for the equations of motion:
\begin{footnotesize}
\begin{equation}
    \label{eq: Klein equations of motion}
    \begin{cases}
    \dot{z}_1 &=-2i\Bigl[-\frac{1}{4\pi}\Gamma_2\frac{\theta'_1\left(\frac{\bar{z}_1 - \bar{z}_2}{2},e^{-\frac{\pi}{2}} + \right)}{\theta_1\left(\frac{\bar{z}_1 - \bar{z}_2}{2},e^{-\frac{\pi}{2}} + \right)}  +\frac{1}{4\pi}\Gamma_2\frac{\theta'_2\left(\frac{\bar{z}_1 - z_2}{2},e^{-\frac{\pi}{2}} + \right)}{\theta_2\left(\frac{\bar{z}_1 - z_2}{2},e^{-\frac{\pi}{2}} + \right)} + \frac{1}{4\pi^2}\Gamma_2\left(\bar{z}_2-z_2\right) + \frac{1}{4\pi}\Gamma_1\Bigl(\frac{\theta_1'\left(\frac{\bar{z}_1-z_1}{2} + \frac{\pi}{2}, e^{-\frac{\pi}{2}}\right)}{\theta_1\left(\frac{\bar{z}_1-z_1}{2} + \frac{\pi}{2}, e^{-\frac{\pi}{2}}\right)} +\\&+ \frac{\bar{z}_1-z_1}{\pi}\Bigr)\Bigr]\\
    \dot{z}_2 &=-2i\Bigl[\frac{1}{4\pi}\Gamma_1\frac{\theta'_1\left(\frac{\bar{z}_1 - \bar{z}_2}{2},e^{-\frac{\pi}{2}} + \right)}{\theta_1\left(\frac{\bar{z}_1 - \bar{z}_2}{2},e^{-\frac{\pi}{2}} + \right)}  +\frac{1}{4\pi}\Gamma_1\frac{\theta'_2\left(\frac{\bar{z}_2 - z_1}{2},e^{-\frac{\pi}{2}} + \right)}{\theta_2\left(\frac{\bar{z}_2 - z_1}{2},e^{-\frac{\pi}{2}} + \right)} + \frac{1}{4\pi^2}\Gamma_1\left(\bar{z}_1-z_1
    \right) + \frac{1}{4\pi}\Gamma_2\Bigl(\frac{\theta_1'\left(\frac{\bar{z}_2-z_2}{2} + \frac{\pi}{2}, e^{-\frac{\pi}{2}}\right)}{\theta_1\left(\frac{\bar{z}_2-z_2}{2} + \frac{\pi}{2}, e^{-\frac{\pi}{2}}\right)} +\\&+ \frac{\bar{z}_2-z_2}{\pi}\Bigr)\Bigr]
    \end{cases}
\end{equation}
\end{footnotesize}
Rewriting (\ref{eq: Ham Klein two vort}) with $x_i$ and $y_i$ rather than $z_i$, one can observe that the Hamiltonian depends on $x_1-x_2$ rather than on $x_1$ and $x_2$ separately. Additionally, we have a constant of motion $C = \Gamma_1y_1 + \Gamma_2y_2$; above we have stressed that this invariant is a local one.

However after making some adjustments to our method, we may treat it as a global one: we suppose that the initial placement of the two vortices is such that their $y-$coordinates lie in the interval $\left[-\frac{\pi}{2},\frac{\pi}{2}\right]$ and proceed to observe this concrete pair of vortices without bounding their $y$-coordinates. In doing so we transfer to a covering system on a cylinder (we will explain this in detail below). The motion of these two vortices defines the motion of entire system, therefore we lose no information.  For the covering  system, the value of  $C$ does not change throughout the motion. Therefore, in following one concrete pair of vortices we may restore our motion while  treating $C$ as a global constant.

 From the periodicity of the Hamiltonian in each $x_i$ we deduce that it is $2\pi$-periodic in $x_1-x_2$; we assume the signs of $\Gamma_1,\Gamma_2$ (supposing, as before, that $\Gamma_1>\Gamma_2>0$) and therefore have to consider the entire period in $x_1-x_2$. Due to this periodicity $x_1-x_2\in (-\pi,\pi)$. Next, we substitute $y_2 = \frac{C}{\Gamma_2} - \frac{\Gamma_1}{\Gamma_2}y_1$ into the Hamiltonian (\ref{eq: Ham Klein two vort}). Since we asume that the values of $y_1$ are unlimited, our covering space will be a cylinder.  By drawing the level sets of the reduced Hamiltonian, we obtain Figure \ref{fig:my_label}

 Note how for the reduced Hamiltonian the  periodicity in $x$-coordinate is retained while  periodicity in $y_1$ is lost -once again, this happens due to $C$ being the local invariant.

\begin{lem}
Critical points of the  reduced Hamiltonian  belong to the lines $x_1-x_2 = 0,\pm \pi$. 
\end{lem}
Critical points of the reduced Hamiltonian correspond to relative equilibria, which we have established to be horizontally moving configurations. Therefore, we need to demonstrate that the flow generated by a solitary vortex is horizontal only on vertical lines that contain the centre of our vortex and the centres of its copies.  In order to do so, we investigate  zeros of $Y_1 :=\mathrm{Re}\left(\frac{\partial \mathcal{H}}{\partial \bar{z_1}}\right)$ after  substituting $\Gamma_1=0$.

\begin{figure}
    \centering
    \subfigure{\includegraphics[scale  = 0.35]{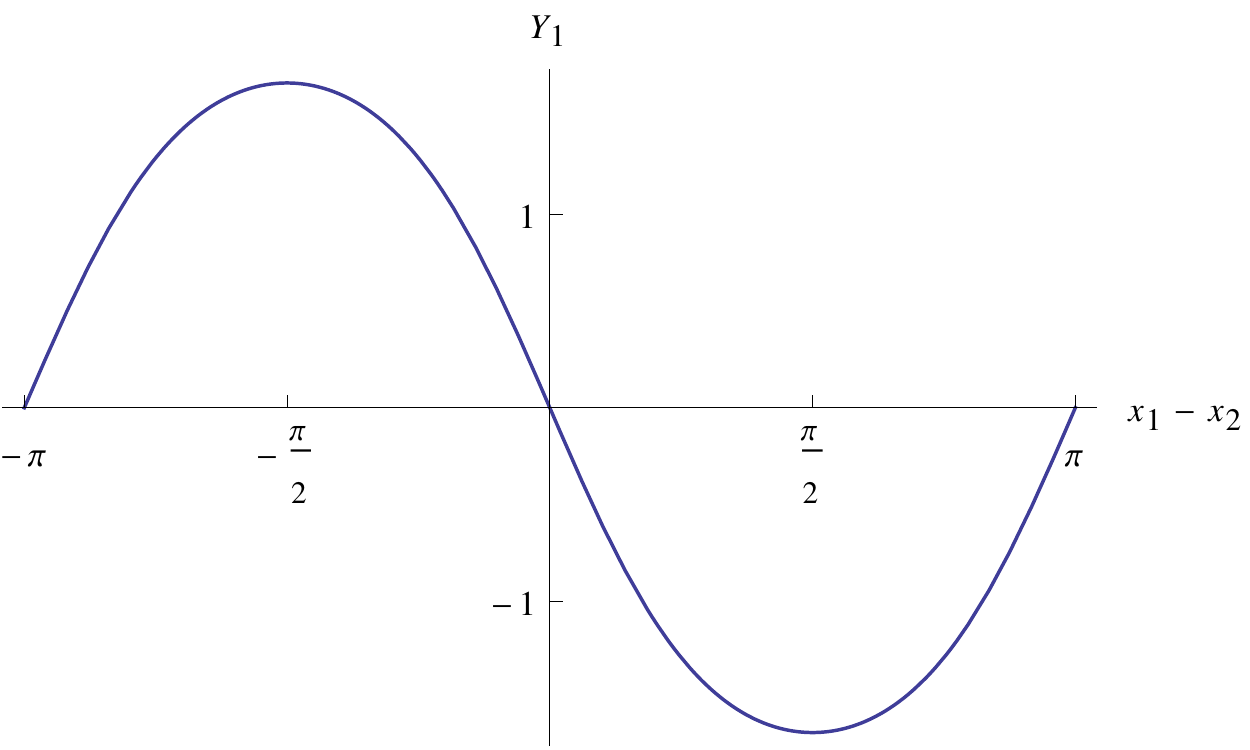}}\ \ 
     \subfigure{\includegraphics[scale  = 0.35]{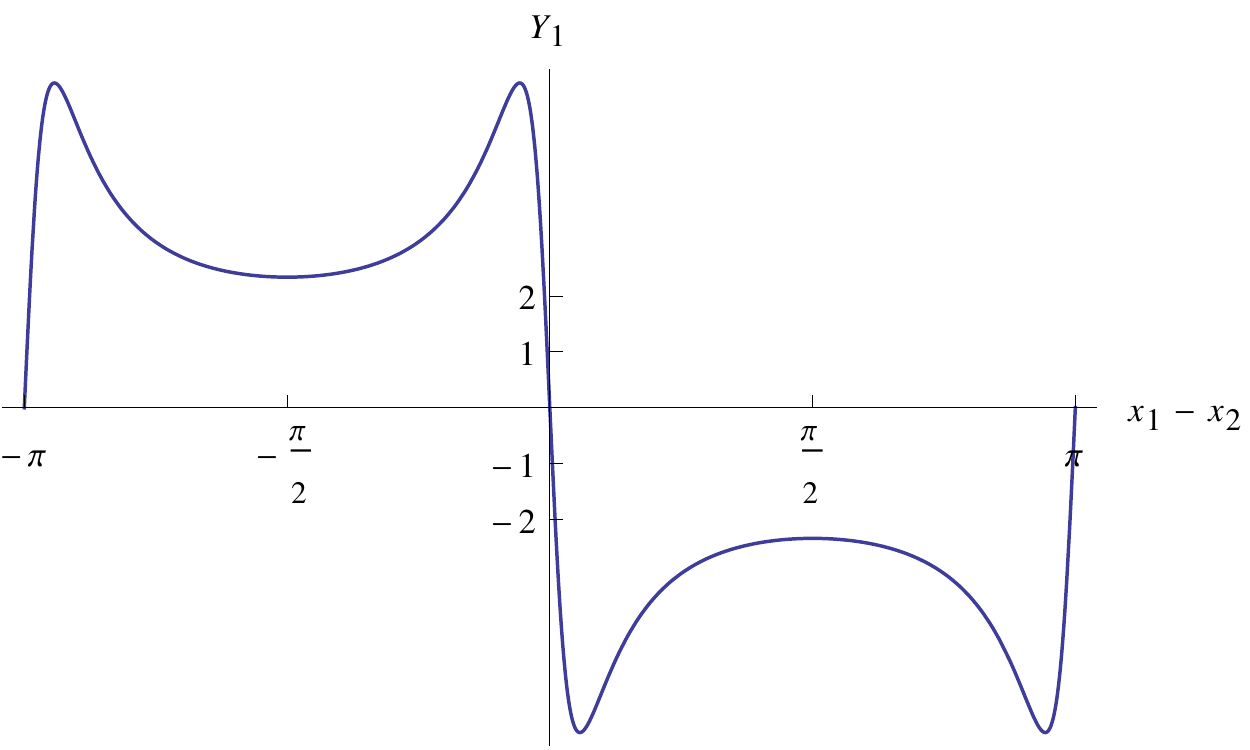}}\ \ 
      \subfigure{\includegraphics[scale  = 0.35]{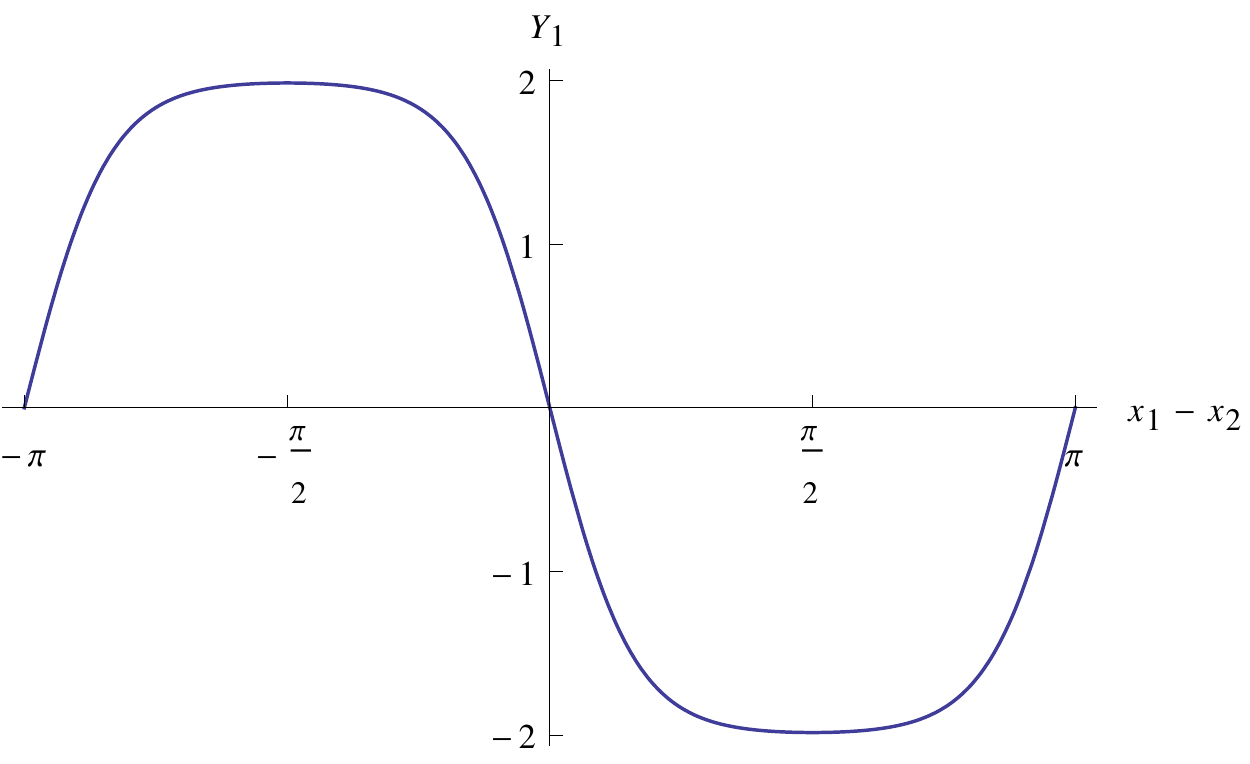}}
    \caption{Three main types of plots of $Y_1$ for fixed $y_1$, $y_2$.  }
    \label{fig: y dot, y not}
\end{figure}
A rigorous proof presents an insurmountable computational challenge; however, numerics demonstrate that for all values of $y_1$ and $y_2$ with $y_1\ne y_2$ the function $Y_1$ turns zero if and only if $x_1-x_2 = 0,\pm \pi$. The three main types of plots for $Y_1$ with fixed values of $y_1$ and $y_2$ are presented in Figure \ref{fig: y dot, y not}. 

  Additionally, the statement of the lemma  can be observed  from Figure \ref{fig: vector field single vortex Klein bottle}(a). The velocity of the flow is inversely proportional to the distance to the vortex; hence, inside the separatrix loops in figure \ref{fig: vector field single vortex Klein bottle}(b) there is no counteracting flow strong enough to change the one induced by the nearest vortex. On the trajectories between the vortices, the vertical components of flows created by nearest vortices have the same sign, and are zero only on the vertical lines to which the centres belong. Hence, the flow will not be horizontal unless on the lines in question.

  

\begin{lem}
The only singular points of the reduced Hamiltonian on the lines $x = 0, \ \pm\pi$ are the points of the form $y_1= \frac{\pi k\Gamma_2 + c}{\Gamma_1 + \Gamma_2}$ and $y_1 = \frac{k\pi\Gamma_2-c}{\Gamma_2 -\Gamma_1}$ respectively, where $k\in\mathbb{Z}$. 
\end{lem}
\begin{proof}
All singularities that the Hamiltonian has are at the configurations where the point vortices collide. When $x_1-x_2 = 0$ and $y_1 = \frac{k\pi\Gamma_2+C}{\Gamma_1 + \Gamma_2}$, $y_2$ is equal to $\frac{C-k\pi\Gamma_1}{\Gamma_1 + \Gamma_2}$. Then the following relation holds:
\[
y_1-k\pi = \frac{C-k\pi\Gamma_1}{\Gamma_1 + \Gamma_2}= y_2,
\]
and the covering copies of the two vortices on the plane collide.

When $x_1-x_2 =\pm\pi$, $y_1 = \frac{k\pi\Gamma_2-C}{\Gamma_2-\Gamma_1}$, which entails $y_2 = \frac{C-k\pi\Gamma_1}{\Gamma_2 - \Gamma_1}$.
Therefore,
\[
y_1 - k\pi = \frac{k\pi\Gamma_1-C}{\Gamma_2-\Gamma_1} = -y_2. 
\]
Similarly to the case of the M\"obius band, the first point vortex in this case collides with a '-' copy of the second one. 

\end{proof}
\begin{theorem}
Consider the level sets of the Hamiltonian, as drawn in  Figure \ref{fig:my_label}. On the closed curves  the motion will be the same as in Regions I and II on the M\"obius band. On non-closed trajectories the global motion will have both vertical and horizontal translational components, and the relative motion of the vortices need not be periodic.
\end{theorem}
\begin{proof}
In here, we rely heavily on the results that we obtained for the case of the M\"obius band in Part I: all the reasoning about smooth and non-trivial dependence of integrals $\int\limits_0^T\dot{x}_k\mathrm{d}t$ on the trajectory can be repeated verbatim. 

We single out three separate cases: trajectories  of Type I are the closed trajectories around singular points on the line $x_1-x_2 = 0$; trajectories of Type II are the closed ones around singularities on $x_1-x_2 = \pm\pi$. 
Trajectories in Region III are the ones that are not closed. 

Analogously to the case of the M\"obius band, we deduce that in Regions I and II the two point vortices rotate around each other and around the bottle; since they emulate the behaviour of a solitary point vortex when in close proximity to each other, a vortex pair will be stationary if and only if $C = 0, \pm\frac{\pi}{4}(\Gamma_1 + \Gamma_2),\pm\frac{\pi}{2}(\Gamma_1  + \Gamma_2)$ (these being the momentum values for which a solitary point vortex is stationary). Otherwise the motion will have a horizontal translational component. 

On the 'vertical' (non-closed) trajectories the two point vortices will move approximately as  point vortices in Region III of the M\"obius band do; but here, as we have mentioned, periodicity in $y_1$ is lost: if we 'glue' the torus according to the periodizations, the trajectories in this region will not necessarily become closed. Additionally, the motion is unbounded in $y_1$, and, therefore, in $y_2$ as well.

Employing the reasoning analogous to the one in the case of the  M\"obius band,  we demonstrate that the horizontal translational component (the integral $\int_0^T\dot{x}_1\mathrm{d}t$) of motion is nonzero for almost all values of $C$;  observe  that the motion very close to  relative equilibria on the line $x_1-x_2=0$ must have nonzero horizontal components.  

Lastly, observe how  $y_1$ is almost a monotonic function on the trajectories in this region; that means that on the induced trajectory on the Klein bottle,  as the time progresses, the vortex will have alternating signs of its $y$-coordinate. 
\end{proof}

\end{document}